\documentclass[12pt,reqno]{article}

\usepackage[usenames]{color}
\usepackage{amssymb}
\usepackage{graphicx}
\usepackage{amscd}

\usepackage[colorlinks=true,
linkcolor=webgreen,
filecolor=webbrown,
citecolor=webgreen]{hyperref}

\definecolor{webgreen}{rgb}{0,.5,0}
\definecolor{webbrown}{rgb}{.6,0,0}

\usepackage{color}
\usepackage{fullpage}
\usepackage{float}

\usepackage{graphics,amsmath,amssymb}
\usepackage{amsthm}
\usepackage{amsfonts}
\usepackage{latexsym}
\usepackage{epsf}

\newcommand{\seqnum}[1]{\href{https://oeis.org/#1}{\underline{#1}}}
\begin{document}
	
	\theoremstyle{plain}
	\newtheorem{theorem}{Theorem}
	\newtheorem{corollary}[theorem]{Corollary}
	\newtheorem{lemma}[theorem]{Lemma}
	\newtheorem{proposition}[theorem]{Proposition}
	\newtheorem{question}[theorem]{Question}
	
	\theoremstyle{definition}
	\newtheorem{definition}[theorem]{Definition}
	\newtheorem{example}[theorem]{Example}
	\newtheorem{conjecture}[theorem]{Conjecture}
	
	\theoremstyle{remark}
	\newtheorem{remark}[theorem]{Remark}
	
	\begin{center}
		\vskip 1cm{\LARGE\bf Linear Combinations of Dirichlet Series Associated with the Thue-Morse Sequence}
		\vskip 1cm
		\large
		L\'aszl\'o T\'oth \\
		Grand Duchy of Luxembourg \\
		\href{mailto:uk.laszlo.toth@gmail.com}{\tt uk.laszlo.toth@gmail.com}
	\end{center}
	
	\vskip .2 in
	
	\begin{abstract}
		In this paper we study sums of Dirichlet series whose coefficients are terms of the Thue-Morse sequence and variations thereof. We find closed-form expressions for such sums in terms of known constants and functions including the Riemann $\zeta$ function and the Dirichlet $\eta$ function using elementary methods.
	\end{abstract}

	\section{Introduction}

Let $t_n$ denote the binary digit sum of the positive integer $n$ modulo $2$, in other words, the $n^{\rm th}$ element of the Thue-Morse sequence $(t_n)_{n\geq0}$, which begins
$$
0,1,1,0,1,0,0,1,\ldots.
$$
This sequence was first considered by Prouhet \cite{Prouhet51} and has applications in many different fields of mathematics (see for instance Allouche and Shallit \cite{AlloucheShallit99} for a detailed overview). A popular variant is the $\pm1$ Thue-Morse sequence, denoted by $\varepsilon_n$ and defined as $((-1)^{t_n})_{n\geq0}$. Dirichlet series associated with $\varepsilon_n$ have been widely studied during the past years, in particular $\sum_{n\geq1} \frac{\varepsilon_n}{n^s}$ and $\sum_{n\geq0} \frac{\varepsilon_n}{(n+1)^s}$, which converge for $\Re(s)>1$. Allouche and Cohen \cite{AlloucheCohen85} continued the series analytically and gave the functional equation
$$
\sum_{n\geq0} \frac{\varepsilon_n}{(n+1)^s} = \sum_{k\geq1} 2^{-s-k} \binom{s+k-1}{k} \sum_{n\geq0} \frac{\varepsilon_n}{(n+1)^{s+k}},
$$
and showed that it admits non-trivial zeros at $s = (2k\pi i)/ \log2$ for any integer $k$. The series $\sum_{n\geq1} \frac{\varepsilon_n}{n^s}$ is continued in a similar manner, yielding a set of non-trivial zeros at $s = i\pi(2k+1)/ \log2$ (although as noted by Allouche \cite{Allouche15}, the question of whether these are \textit{all} the non-trivial zeros is still an open one). These results were then further extended by Allouche, Mend\`es France and Peyri\`ere \cite{Allouche00} to $d$-automatic sequences, for $d\geq2$. Furthermore, Allouche and Cohen \cite{AlloucheCohen85} and Alkauskas \cite{Alkauskas01} have noticed that the two series above are related through the identity
$$
\sum_{n\geq0} \frac{\varepsilon_n}{(n+1)^s} = \frac{1-2^s}{1+2^s} \sum_{n\geq1} \frac{\varepsilon_n}{n^s}.
$$
Since then, these series have been used in various contexts. For instance, Allouche and Cohen \cite{AlloucheCohen85} used the derivative of $\sum_{n\geq0} \frac{\varepsilon_n}{(n+1)^s}$ at $s=0$ to give an alternative proof of the Woods-Robbins product,
$$
\prod_{n\geq0} \left( \frac{2n+1}{2n+2} \right) ^{\varepsilon_n} = \frac{\sqrt2}{2},
$$
which has been extended to many different types of infinite products (see Allouche, Cohen, Mend\`es France and Shallit \cite{AlloucheShallit87}, Allouche and Sondow \cite{AlloucheSondow08}, Allouche, Riasat and Shallit \cite{Riasat19} and T\'oth \cite{Toth20} for example).

Series involving the $b$-ary sum-of-digits function $s_b(n)$ have also been studied. A classic example is
$$
\sum_{n\geq1} \frac{s_b(n)}{n(n+1)} = \frac{b}{b-1} \log b,
$$
which was proved by Shallit \cite{Shallit84}, while other identities such as
$$
\sum_{n\geq1} \frac{s_2(n) (2n+1)}{n^2(n+1)^2} = \frac{\pi^2}{9}
$$
are due to Allouche and Shallit \cite{AlloucheShallit90}.

\subsection{Scope of this paper}

No closed-form expression is currently known for Dirichlet series involving only $t_n$ and $t_{n-1}$ and combinations thereof in terms of known constants and functions. In this paper, we provide such closed forms and give a  formula for generating similar linear combinations. A few examples are
\begin{align*}
	\sum_{n\geq1} \frac{5 t_{n-1} + 3 t_n}{n^2} &= \frac{2 \pi^2}{3}, \\
	\sum_{n\geq1} \frac{9 t_{n-1} + 7 t_n}{n^3} &= 8 \zeta(3).
\end{align*}
Then, we extend our results to combinations of sequences involving different alphabets $\{a,b\}$ along the Thue-Morse sequence, that is, sequences with the substitutions $0\to a$ and $1\to b$ over $t_n$, with $a,b\in\mathbb{R}$. An example is
$$
\sum_{n\geq1} \frac{17 q_{n-1} + 15 r_n}{n^4} = 16 \eta(4),
$$
where $\eta(s)$ denotes the alternating zeta function $\sum_{n\geq1} \frac{(-1)^{n-1}}{n^s}$ and $q_n$ and $r_n$ denote the sequences defined respectively by $\{-\sqrt2,1-\sqrt2\}$ and $\{\frac{17\sqrt2 -2}{15},\frac{17\sqrt2 +13}{15}\}$ along the Thue-Morse sequence.

\section{Linear combinations}
Throughout the remainder of this paper, we shall use the following notation:
$$
f(s)=\sum_{n\geq1} \frac{\varepsilon_{n-1}}{n^s}, \ \ \ \  g(s) = \sum_{n\geq1} \frac{\varepsilon_n}{n^s}, \ \ \ \  \varphi(s) = \sum_{n\geq1} \frac{t_{n-1}}{n^s}, \ \ \ \  \gamma(s) = \sum_{n\geq1} \frac{t_n}{n^s}.
$$

We begin this section by recalling a result on Dirichlet series whose sign alternates following the Thue-Morse sequence, originally due to Allouche and Cohen \cite{AlloucheCohen85} and already mentioned in the introduction.
\begin{lemma} \label{lemma-alkauskas}
	For all $s \in \mathbb{C}$ with $\Re(s)>1$,
	$$
	f(s) = \frac{1-2^s}{1+2^s} g(s).
	$$
\end{lemma}
While Allouche and Cohen's proof involves the analytic continuations of $f$ and $g$, a simpler proof was proposed by Alkauskas \cite{Alkauskas01} in 2001 that we cannot resist replicating here. This proof relies only on the fact that every positive integer $n$ can be uniquely represented as the product $2^k(2m+1)$ for $k\geq0$ and $m\geq0$.
\begin{proof} [Proof of Lemma 1 \rm (\cite{Alkauskas01})] 
	We have
	$$
	f(s) = \sum_{k\geq0,m\geq0} \frac{\varepsilon_{2^k(2m+1)-1}}{2^{ks}(2m+1)^s} = \sum_{k\geq0,m\geq0} \frac{(-1)^k\varepsilon_m}{2^{ks}(2m+1)^s} = \frac{2^s}{2^s+1} \sum_{m\geq0} \frac{\varepsilon_m}{(2m+1)^s}.
	$$
	Now we know that
	$$
	\sum_{m\geq0} \frac{\varepsilon_m}{(2m+1)^s} = \sum_{m\geq1} \frac{\varepsilon_m}{(2m)^s} - \sum_{m\geq1} \frac{\varepsilon_m}{m^s},
	$$
	by splitting $\sum_{m\geq1} \frac{\varepsilon_m}{m^s}$ into even and odd indexes and using the fact that $\varepsilon_{2m}=\varepsilon_m$ and $\varepsilon_{2m+1}=-\varepsilon_m$. The proof follows naturally.
\end{proof}
We note that one could also evaluate $g(s)$ in the same manner, thus yielding a unified proof of Lemma \ref{lemma-alkauskas}. This is left as an exercise for the reader.

\begin{corollary} \label{coro-holo}
	For any holomorphic functions $u$ and $v$, we have
	$$
	u(s)\varphi(s) + v(s)\gamma(s) = \frac{u(s) + v(s)}{2} \zeta(s) - \frac{f(s)}{2} \left(u(s)+v(s)\frac{1+2^s}{1-2^s}\right).
	$$
\end{corollary}
\begin{proof}
	Using Lemma \ref{lemma-alkauskas} together with $\varepsilon_n = 1 - 2t_n$, we have
	$$
	\varphi(s) = \frac12 \zeta(s) - \frac12 f(s), \ \ \ \ \gamma(s) = \frac12 \zeta(s) - \frac{1+2^s}{2(1-2^s)} f(s).
	$$
	Now taking two holomorphic functions $u$ and $v$, we quickly obtain the above expression for $u(s) \varphi(s) + v(s) \gamma(s)$.
\end{proof}

We now have all the tools to establish our first result in this paper.
\begin{theorem} \label{theo-riemann-zeta}
	Let $\zeta(s)=\sum_{n\geq1} 1/n^s$ denote the Riemann zeta function defined for complex $s$ with $\Re(s)>1$. Then
	$$
	(2^s+1) \varphi(s) + (2^s-1) \gamma(s) = 2^s \zeta(s).
	$$
\end{theorem}
Note that both Dirichlet series on the left-hand side converge for $\Re(s)>1$ since the sequence $(t_n)_{n\geq0}$ takes only finitely many values. 

We shall now give two proofs of this identity. The first uses Corollary \ref{coro-holo}, while the second relies on index-splitting (used also within the context of the Woods-Robbins product by Allouche, Mend\`es France and Peyri\`ere \cite{Allouche00}, for instance).
\begin{proof} [Proof 1 of Theorem \ref{theo-riemann-zeta}]
	We take Corollary \ref{coro-holo} with $u(s) = 2^s+1$ and $v(s) = 2^s-1$. The proof immediately follows.
\end{proof}
Of course, we could also prove the Theorem above without having recourse to any of the results above. For instance, we could simply use the relations $t_{2n} = t_n$ and $t_{2n+1} = 1-t_n$ as follows.
\begin{proof}[Proof 2 of Theorem \ref{theo-riemann-zeta}]
	We begin by splitting $\gamma(s)$ and $\varphi(s)$ into odd and even indexes. On the one hand, we have
	$$
	\gamma(s) = 2^{-s} \gamma(s) - \sum_{n\geq0} \frac{t_{n}}{(2n+1)^s} + (1-2^{-s}) \zeta(s).
	$$
	On the other hand,
	$$
	\varphi(s) = 2^{-s} \zeta(s) - 2^{-s} \varphi(s) + \sum_{n\geq0} \frac{t_{n}}{(2n+1)^s}.
	$$
	Now taking the sum of these two equations yields
	$$
	\gamma(s) + \varphi(s) =  2^{-s} ( \gamma(s) - \varphi(s) ) + \zeta(s),
	$$
	and a simple rearrangement of the terms concludes this proof.
\end{proof}
Here we note that whenever $t_{n-1}=t_n=0$, the corresponding $n^{\rm th}$ term disappears from both Dirichlet series on the left-hand side. The first few missing terms are thus $n=6,10,18,24,\ldots$ (sequence A248056 in the OEIS). Our result above implies several interesting examples, which we have already mentioned in the introduction.
\begin{example}
	\rm{We have the following equalities:}
	\begin{align*}
		(a) \ \ \ \  \ \ \ \ &\sum_{n\geq1} \frac{5 t_{n-1} + 3 t_n}{n^2} = \frac{2 \pi^2}{3}, \\
		(b) \ \ \ \  \ \ \ \ &\sum_{n\geq1} \frac{9 t_{n-1} + 7 t_n}{n^3} = 8 \zeta(3).
	\end{align*}
\end{example}
There are several ways to extend these results, which we will do in the following sections.

\subsection{Generalization to different alphabets along the Thue-Morse sequence}
In the following paragraphs, we extend Theorem \ref{theo-riemann-zeta} to linear combinations of series involving various alphabets $\{a,b\}$ along the Thue-Morse sequence, i.e., sequences with the substitutions $0\to a$ and $1\to b$, with $a,b\in\mathbb{R}$, over the $t_n$ sequence. In particular, we show that there exist alphabets which give rise to identities involving the alternating zeta function $\eta(s) = \sum_{n\geq1} \frac{(-1)^{n-1}}{n^s}$, others that give rise to identities involving the Riemann $\zeta$ function and, finally, alphabets that produce simple linear combinations with other Dirichlet series.
\begin{theorem} \label{theo-alphabets}
	Let $k,\ell>0$ be real numbers and define the sequences $q_n = t_n - k$ and $r_n = t_n + \ell$ for all $n>0$. Furthermore, consider the function $\lambda(s;k,\ell) = 2^s - (2^s (k-\ell) + (k+\ell))$ for some complex $s$ with $\Re(s)>1$. We then have
	\begin{itemize}
		\item[\rm (1)] {\makebox[8cm]{$\displaystyle \sum_{n\geq1} \frac{q_{n-1}}{n^s} = \frac{1-2^s}{1+2^s} \sum_{n\geq1} \frac{r_n}{n^s}$\hfill} if $\lambda(s;k,\ell) = 0$},
		\item[\rm (2)] {\makebox[8cm]{$\displaystyle(2^s+1) \sum_{n\geq1} \frac{q_{n-1}}{n^s} + (2^s-1) \sum_{n\geq1} \frac{r_n}{n^s} = 2^s \zeta(s)$\hfill} if $\lambda(s;k,\ell) = 2^s$,}
		\item[\rm (3)] {\makebox[8cm]{$\displaystyle(2^s+1) \sum_{n\geq1} \frac{q_{n-1}}{n^s} + (2^s-1) \sum_{n\geq1} \frac{r_n}{n^s} = 2^s \eta(s)$\hfill} if $\lambda(s;k,\ell)=2^s-2$}.
	\end{itemize}
\end{theorem}
\begin{proof}
	Define the sequences $q_n = t_n - k$ and $r_n = t_n + \ell$ for all $n>0$ and real $k,\ell>0$. We immediately have
	$$
	\sum_{n\geq1} \frac{q_{n-1}}{n^s} = \varphi(s) -k\zeta(s), \ \ \ \ \sum_{n\geq1} \frac{r_{n}}{n^s} = \gamma(s) +l\zeta(s).
	$$
	Thus,
	$$
	(2^s+1) \sum_{n\geq1} \frac{q_{n-1}}{n^s} + (2^s-1) \sum_{n\geq1} \frac{r_n}{n^s} = \zeta(s) \left(2^s - (2^s (k-l) + (k+l))\right) .
	$$
	Inspired by the coefficient of $\zeta(s)$ on the right-hand side, we define the function $\lambda(s;k,\ell) = 2^s - (2^s (k-\ell) + (k+\ell))$ for real $k,\ell>0$ and complex $s$ with $\Re(s)>1$. It is obvious that if $\lambda(s;k,\ell)=0$, the $\zeta(s)$ term on the right-hand side vanishes, thus proving (1). If $\lambda(s;k,\ell)=2^s$, we have identities of the same type as in Theorem \ref{theo-riemann-zeta}, proving (2), and if $\lambda(s;k,\ell) = 2^s-2$ then the identity $\eta(s) = \left(1-2^{1-s}\right) \zeta(s)$ quickly establishes (3).
\end{proof}
Solutions to each of these cases lead to interesting illustrative examples. A first set of ``simple'' solutions -- in terms of $k$ and $\ell$ only -- are easy to find. For the case (1) above, we have $k=\frac12, \ell=-\frac12$, which leads to the classic Woods-Robbins product by differentiation of the resulting series identity at $s=0$, as already noted by Allouche and Cohen \cite{AlloucheCohen85}. The solution of the second case, $k=0, \ell=0$, results in our identity in Theorem \ref{theo-riemann-zeta}, and finally that of the third case ($k=1, \ell=1$) results in
$$
(2^s+1) \sum_{n\geq1} \frac{q_{n-1}}{n^s} + (2^s-1) \sum_{n\geq1} \frac{r_n}{n^s} = 2^s \eta(s),
$$
where $q_n \to \{-1,0\}$ and $r_n \to \{1,2\}$ along the Thue-Morse sequence.

Another, perhaps more interesting set of solutions can be found by taking $s$ into account as well.

\begin{proposition}
	Define the sequences $q_n = t_n - k$ and $r_n = t_n + \ell$ for all $n>0$ and real $k,\ell>0$. We have the following equalities:
	\begin{align*}
		(a)& \ \ \ \  \ \ \ \ 5 \sum_{n\geq1} \frac{q_{n-1}}{n^2} = -3 \sum_{n\geq1} \frac{r_n}{n^2}, & & q_n \to \{-1,0\}, r_n \to \{\frac13,\frac43\} \\
		(b)&  \ \ \ \  \ \ \ \ \sum_{n\geq1} \frac{9 q_{n-1} + 7 r_n}{n^3} = 8 \zeta(3), & & q_n \to \{-1,0\}, r_n\to \{\frac97,\frac{16}{7}\} \\
		(c)& \ \ \ \  \ \ \ \ \sum_{n\geq1} \frac{17 q_{n-1} + 15 r_n}{n^4} = 16 \eta(4), & & q_n\to\{-\sqrt2,1-\sqrt2\}, \\
		& & & r_n\to\{\frac{17\sqrt2 -2}{15},\frac{17\sqrt2 +13}{15}\}.
	\end{align*}
\end{proposition}
\begin{proof}
	The general solution of the equation $\lambda(s;k,\ell)=0$ (i.e., case (1) in Theorem \ref{theo-alphabets}) in the reals is
	$$
	s=\frac{\log(\frac{k+\ell}{-k+\ell+1})}{\log2},
	$$
	with $k\neq \ell+1$ and $k+\ell\neq 0$. So for instance we can take $k=1$ and $\ell=\frac13$, yielding the  sequences $q_n=t_n-1$ and $r_n=t_n+\frac13$, for all $n\geq0$, i.e., the sequences defined respectively by $\{-1,0\}$ and $\{\frac13,\frac43\}$ along the Thue-Morse sequence. A simple substitution above gives $s=2$, thereby proving our statement (a).
	
	We now turn our attention to the solution of $\lambda(s;k,\ell)=2^s$ (i.e., case (2) in Theorem \ref{theo-alphabets}), which in the reals is
	$$
	s=\frac{\log(-\frac{k+\ell}{k-\ell})}{\log2},
	$$
	with $k-\ell\neq 0$ and $k+\ell\neq 0$. Let now $q_n=t_n-1$ and $r_n=t_n+\frac97$, for all $n\geq0$, in other words the sequences defined respectively by $\{-1,0\}$ and $\{\frac97,\frac{16}{7}\}$ along the Thue-Morse sequence. This means that we have $k=1$ and $\ell=\frac97$, which yields $s=3$ in the equation above and thus proves statement (b).
	
	Finally, the solution of the equation $\lambda(s;k,\ell)=2^s-2$ in the reals, corresponding to case (3) in Theorem \ref{theo-alphabets} is
	$$
	s=\frac{\log(-\frac{k+\ell-2}{k-\ell})}{\log2},
	$$
	with $k-\ell\neq 0$ and $k+\ell\neq 2$, which allows us to choose $k=\sqrt2$ and $\ell=\frac{17\sqrt2 -2}{15}$. This gives the sequences $q_n=t_n-\sqrt2$ and $r_n=t_n+\frac{17\sqrt2 -2}{15}$, for all $n\geq0$, i.e., the sequences defined respectively by $\{-\sqrt2,1-\sqrt2\}$ and $\{\frac{17\sqrt2 -2}{15},\frac{17\sqrt2 +13}{15}\}$ along the Thue-Morse sequence. Thus we have $s=4$ and identity (c) as claimed.
\end{proof}

\section{Conclusion and further work}

In this paper we have found closed forms for certain linear combinations of Dirichlet series associated with the Thue-Morse sequence in terms of known constants and functions. However, closed forms for the individual series remain elusive.
\begin{question}
	Do the series $\displaystyle \sum_{n\geq1} \frac{t_n}{n^s}$ and $\displaystyle \sum_{n\geq1} \frac{t_{n-1}}{n^s}$ for $s \in \mathbb{C}$ with $\Re(s)>1$ admit closed forms in terms of known constants and functions?
\end{question}
Despite our efforts, we have not been able to find a set of linear combinations allowing us to eliminate either $f$ or $g$ from Theorem \ref{theo-riemann-zeta}.

\subsection{Extension to other sequences}

In some of our proofs we used an index-splitting method to find expressions for series involving the Thue-Morse sequence and variations thereof. The same method can possibly be applied to other series whose coefficients are generated by finite automata. A few examples are listed below, and the proofs are left to the reader.
\begin{example} \label{theo-delta} \rm
	Let $\delta_n = t_n-t_{n-1}$ for all $n\geq1$ and $\varepsilon_n=(-1)^{t_n}$ the $\pm1$ Thue-Morse sequence. Then for all $s$ with $\Re(s)>1$,
	$$
	\sum_{n\geq1} \frac{\delta_n}{n^s} = \frac{4^s}{4^{s}-1} \sum_{n\geq0} \frac{\varepsilon_n}{(2n+1)^s}.
	$$
\end{example}

\begin{example} \rm
	Let $\sigma_n$ denote the ``period-doubling sequence'' (A096268 in the OEIS), defined by the recurrence $\sigma_{2n}=0, \sigma_{4n+1} = 1, \sigma_{4n+3} = \sigma_n$. Then for all $s$ with $\Re(s)>1$,
	$$
	\sum_{n\geq1} \frac{\sigma_n \left( (4n+3)^{s} - n^{s} \right)}{\left( 4n^2+3n \right)^s} = 4^{-s} \zeta\left(s,\frac14\right),
	$$
	where $\zeta(s,a)$ denotes the Hurwitz zeta function.
\end{example}

\noindent {\bf Acknowledgement.}
The author is indebted to two anonymous referees, whose comments have greatly increased the quality of this paper.

\bigskip
\hrule
\bigskip

\noindent 2010 {\it Mathematics Subject Classification}: Primary 11A63; Secondary 68R15, 11Y60, 11B85, 11M41. \\
\noindent \emph{Keywords: } Thue-Morse sequence, Dirichlet series, Riemann zeta function.

\bigskip
\hrule
\bigskip

\noindent (Concerned with sequences
\seqnum{A000120},
\seqnum{A010060},
\seqnum{A014081},
\seqnum{A020985},
\seqnum{A029883},
\seqnum{A088705},
\seqnum{A096268},
\seqnum{A106400}, and
\seqnum{A248056}.)

\bigskip
\hrule
\bigskip
	
\end{document}